\documentclass[a4paper,10pt]{amsart}

\usepackage{amssymb,amsmath,amsfonts,amsthm,mathrsfs,mathtools,amscd,tikz-cd,dirtytalk,comment,bbm}
\usepackage[hidelinks]{hyperref}
\usepackage[capitalise]{cleveref}

\newtheorem{theorem}{Theorem}[section]
\newtheorem{lemma}[theorem]{Lemma}
\newtheorem{proposition}[theorem]{Proposition}
\newtheorem{corollary}[theorem]{Corollary}

\theoremstyle{definition}
\newtheorem{definition}[theorem]{Definition}
\newtheorem{example}[theorem]{Example}

\theoremstyle{remark}
\newtheorem{remark}[theorem]{Remark}

\newtheorem{question}[theorem]{Question}

\numberwithin{equation}{section}

\DeclareMathOperator{\Aut}{Aut}

\DeclareMathOperator{\End}{End}

\DeclareMathOperator{\Hol}{Hol}

\DeclareMathOperator{\id}{id}

\DeclareMathOperator{\Soc}{Soc}

\DeclareMathOperator{\ad}{ad}

\DeclareMathOperator{\Der}{Der}

\DeclareMathOperator{\Laz}{\mathbf{Laz}}
\DeclareMathOperator{\aff}{\mathfrak{aff}}
\DeclareMathOperator{\BCH}{BCH}
\DeclareMathOperator{\Fix}{Fix}

\newcommand{\fg}{\mathfrak{g}}
\newcommand{\fh}{\mathfrak{h}}
\newcommand{\fa}{\mathfrak{a}}
\newcommand{\fb}{\mathfrak{b}}
\newcommand{\tr}{\triangleright}
\newcommand{\cL}{\mathcal{L}}

\newcommand{\Z}{\mathbb{Z}}

\newcommand{\F}{\mathbb{F}}

\newcommand{\bS}{\mathbf{S}}
\newcommand{\bL}{\mathbf{L}}
\DeclareMathOperator{\Ann}{Ann}
\DeclareMathOperator{\pr}{pr}
\newcommand{\U}{\mathcal{U}}
\newcommand{\hU}{\hat{\mathcal{U}}}

\begin{document}

\title{A Lazard correspondence for post-Lie rings and skew braces}
\author{S.~Trappeniers}

\address{Department of Mathematics and Data Science, 
 Vrije Universiteit Brussel, 
 Pleinlaan 2, 
 1050 Brussels, Belgium} 

\email{senne.trappeniers@vub.be} 
\urladdr{https://sites.google.com/view/sennetrappeniers/}

\begin{abstract}
    We develop a Lazard correspondence between post-Lie rings and skew braces that satisfy a natural completeness condition. This is done through a thorough study of how the Lazard correspondence behaves on semi-direct sums of Lie rings. In particular, for a prime $p$ and $k<p$, we obtain a correspondence between skew braces of order $p^k$ and left nilpotent post-Lie rings of order $p^k$ on a nilpotent Lie ring. This therefore extends results by Smoktunowicz.
\end{abstract}
\maketitle

\section{Introduction}
In 1954, Lazard proved a correspondence between Lie rings and groups satisfying certain completeness and divisibility conditions \cite{Lazard1954SurLie}. This built upon earlier work by Malcev \cite{Malcev1949}. The nowadays better known version of the Lazard correspondence, although more restrictive than the original version, states that for $p^k$ a prime power, there is a bijective correspondence between nilpotent Lie rings of order $p^k$ and nilpotency class at most $p-1$, and nilpotent groups of order $p^k$ of nilpotency class at most $p-1$. In particular, one finds a correspondence between groups of order $p^k$ and nilpotent Lie rings of order $p^k$ for $k<p$. This has proved to be very useful in the classification of finite $p$-groups, see for example \cite{Newman2004GroupsPrime}. 

Pre-Lie algebras, also called left-symmetric algebras, arose as early as in 1881 in the work of Cayley \cite{Cayley1881}, only to resurface again 80 years later in the works of Vinberg \cite{Vinberg1963} and Koszul \cite{Koszul1961}. See \cite{Burde2006Left-symmetricPhysics} for a survey on pre-Lie algebras. One of the motivations to study pre-Lie algebras comes from geometry, as they appear naturally in the study of regular affine actions of Lie groups \cite{Kim1986CompleteGroups}. More recently, the notion of a post-Lie algebra was introduced in the works of Vallette \cite{Vallette2007HomologyPosets}, see also \cite{Curry2019WhatIntegration}. For $R$ a commutative ring, a \emph{post-Lie $R$-algebra} is a Lie $R$-algebra $\mathfrak{a}$, whose Lie bracket we denote by $[-,-]$, together with an $R$-bilinear map $\triangleright: \mathfrak{a}\times \mathfrak{a}\to \mathfrak{a}$ satisfying:
    \begin{align}
        x\tr[y,z]&=[x\tr y,z]+[y,x\tr z],\label{eq: postLie 1}\\\
        [x,y]\tr z&=(x,y,z)_\tr-(y,x,z)_\tr,\label{eq: postLie 2}
    \end{align}
    with $(x,y,z)_\tr:=x\tr (y\tr z)-(x\tr y)\tr z$. If $\fa$ is abelian, then $(\fa,\tr)$ is a \emph{pre-Lie $R$-algebra}. A post-Lie $\Z$-algebra, respectively pre-Lie $\Z$-algebra, is called a \emph{post-Lie ring}, respectively a \emph{pre-Lie ring}. For a post-Lie $R$-algebra $(\fa,\tr)$, define 
    \begin{equation}\label{eq: a^circ}
        \{a,b\}:=[a,b]+a\tr b-b\tr a.
    \end{equation}
    It follows directly from \eqref{eq: postLie 1} and \eqref{eq: postLie 2} that this defines a second Lie $R$-algebra structure on the underlying $R$-algebra $\fa$. We will denote this Lie $R$-algebra by $\fa^\circ$.

Braces were introduced in 2006 by Rump \cite{rump2006modules,rump2007braces} in the study of the set-theoretical solutions to the Yang--Baxter equation. They were subsequently generalized by Guarnieri and Vendramin to skew braces \cite{guarnieri2017skew}. 
A \emph{skew brace} is a triple $(A,\cdot,\circ)$ with $A$ a set and $(A,\cdot)$ and $(A,\circ)$ group structures such that 
$$a\circ (b\cdot c)=(a\circ b)\cdot a^{-1}\cdot (a\circ b),$$
holds for all $a,b,c\in A$. Here $a^{-1}$ denotes the inverse of $a$ in $(A,\cdot)$. When $(A,\cdot)$ is abelian, then $(A,\cdot,\circ)$ is called a \emph{brace}. For braces, the notation $(A,+,\circ)$ is also common (in fact, this is also the case for skew braces, but to avoid any possible confusion we will abstain from using $+$ for a non-abelian group operation).
Skew braces are essentially the same as bijective 1-cocycles of groups or regular subgroups of the holomorph of a group. In this sense, they were already being studied in a geometric framework in the form of regular affine actions of Lie groups on vector spaces \cite{Auslander1977SimplyTransitive,Milnor1977OnFundamental} long before they were formally introduced. Since pre-Lie algebras played an important role there, some connection between skew braces and pre-Lie rings is to be expected. This was indeed first stated by Rump \cite{Rump2014TheGroup} for $\mathbb{R}$-braces. Bachiller then used the Lazard correspondence to construct a counterexample of a conjecture on braces \cite{Bachiller2016}, starting from a counterexample by Burde \cite{Burde1996AffineNilmanifolds} to a question by Milnor \cite{Milnor1977OnFundamental} (note that the first counterexample to Milnor's question was given earlier by Benoist \cite{Benoist1995}). 

Smoktunowicz proved that the construction of the group of flows, as defined for pre-Lie algebras in \cite{Agrachev1981ChronologicalFields}, could be adapted to left nilpotent pre-Lie rings of prime power order $p^k$, with $k+1<p$, in order to obtain braces \cite{Smoktunowicz2022OnRings} with the same additive structure. Here, a post-Lie ring $(\fa,\tr)$ is \emph{left nilpotent} if there exists some $n$ such that $\fa^n=0$ where $\fa^1=\fa$ and $\fa^{i+1}$ is the subgroup of $(\fa,+)$ generated by $a\tr b$ for $a\in \fa$, $b\in \fa^i$, \emph{right nilpotency} is defined analogously. Similarly, a skew brace $(A,\cdot,\circ)$ is \emph{left nilpotent} if there exists some $n$ such that $A^n=1$, where $A^1=A$ and $A^{i+1}$ is the subgroup of $(A,\cdot)$ generated by the elements $a*b=a^{-1}\cdot(a\circ b)\cdot b^{-1}$ for all $a\in A$, $b\in A^{i}$, \emph{right nilpotency} is defined analogously. Recall that a skew brace of prime power order is always left nilpotent \cite{CSV19}. It is not immediately clear whether Smoktunowicz's construction is invertible, an inverse construction is given only in the case of strongly nilpotent braces of order $p^k$, with $k+1<p$, and with strong nilpotency index at most $p-1$. Here strong nilpotency is equivalent to the structures being both left and right nilpotent \cite{smoktunowicz2018engel}. This correspondence was used to classify all strongly nilpotent braces of order $p^4$, for $p>5^5$, in \cite{Puljic2022Classificationp4} while the non-strongly nilpotent ones of this order were classified earlier in \cite{Puljic2022SomeExtensions}.  In \cite{Shalev2024FromBraces}, Shalev and Smoktunowicz give a construction that starts from a brace $(A,+,\circ)$ of order $p^k$, with $k+1<p$, in order to obtain a pre-Lie ring with additive group $(A,+)/\{a\in A\mid p^2a=0\}$. Although the strong nilpotency condition is no longer present, one can not hope that this yields a correspondence as a quotient by a non-trivial subgroup is taken. 

Recently, Bai, Guo, Sheng and Tang showed how from a regular affine action of a Lie group onto another Lie group, one obtains a post-Lie algebra through differentiation \cite{Bai2023Post-groupsEquation} (note that, although not explicitly stated in this way, this construction also appears in \cite{Burde2009AffineGroups}). This extended the earlier techniques to obtain a pre-Lie algebra in the case that the latter group is $\mathbb{R}^n$. Conversely, also in \cite{Bai2023Post-groupsEquation} a construction by formal integration was given to construct a skew brace starting from a post-Lie algebra of characteristic 0 and with some completeness conditions.

In \cref{section: preliminaries} we recall the necessary information on the Lazard correspondence, and some basic notions on post-Lie rings and skew braces. In \cref{section: post lie,section: skew} we introduce the notion of filtered post-Lie rings and filtered skew braces. For both structures we define the property of being Lazard in a natural way. We construct a natural filtration on any post-Lie ring or skew brace. As a consequence, we prove that if $(A,\cdot,\circ)$ is a left nilpotent skew brace and $(A,\cdot)$ is a nilpotent group, then also $(A,\circ)$ is a nilpotent group; see \cref{lem: left nilpotent nilpotent type brace} and \cref{theorem: nilpotency class multiplicative skew brace}. This result was previously proved by Ced\'o, Smoktunowicz and Vendramin for finite skew braces \cite[Theorem 4.8]{CSV19}. The introduction of a new notion of nilpotency for skew braces, coined $L$-nilpotency, is essential in our proof and also provides a natural upper bound on the nilpotency class of $(A,\circ)$ in this case.

In \cref{section: lazard generalized} we first prove that the exponential and logarithmic maps give a correspondence between certain derivations and automorphisms of Lazard Lie rings, mimicking the geometric correspondence given by integration and differentiation. We then prove the final correspondence between Lazard post-Lie rings and Lazard skew braces, its most general formulation is given in \cref{theorem: final correspondence} and the version for finite structures in \cref{theorem: correspondence p}.
A more restrictive version of this correspondence, although arguably the most interesting for classification purposes, is the following.
\begin{theorem}\label{theorem: correspondence finite p}
    Let $p$ be a prime and $k<p$. There is a bijective correspondence between skew braces of size $p^k$ and left nilpotent post-Lie rings of size $p^k$ whose underlying Lie ring is nilpotent. This correspondence restricts to braces of size $p^k$ and left nilpotent pre-Lie rings of size $p^k$.
\end{theorem}

This correspondence is explicitly described in \cref{prop: group of flows,prop: back to Lie}. In particular, the brace associated to a pre-Lie ring is precisely the one obtained by the group of flows construction as in \cite{Smoktunowicz2022OnRings}. Therefore, \cref{theorem: correspondence finite p} effectively extends Smoktunowicz's correspondence \cite[Theorem 14]{Smoktunowicz2022OnRings} which imposed stronger nilpotency conditions. This provides an affirmative answer to \cite[Question 1]{Smoktunowicz2022OnRings} and \cite[Problem 20.92 b)]{Kourovka2024}.
When \cref{prop: group of flows} is applied to post-Lie algebras over a field of characteristic $0$, one recovers the formal integration by Bai, Guo, Sheng and Tang \cite[Theorem 5.22]{Bai2023Post-groupsEquation} with less restrictive conditions.

In \cite[Section 5.2]{Bai2023Post-groupsEquation} the authors propose the problem of finding a suitable differentiation and formal integration theory for finite skew braces and post-Lie rings. One such approach is precisely the one using the exponential and the logarithmic map as in \cref{section: lazard generalized}. On the other hand, a technique used by Smoktunowicz in \cite{Smoktunowicz2022AAlgebras,Smoktunowicz2022OnRings} could also be seen as an discrete differentiation. In \cref{section: differentiation} we motivate this and we give an application to skew braces that extends a result of Smoktunowicz. 

\section{Preliminaries}\label{section: preliminaries}
We first recall the precise statement of the Lazard correspondence together with some useful related results. Although we use slightly different terminology, all of the statements regarding the Lazard correspondence follow from the treatment in \cite{Khukhro1998,Lazard1954SurLie}.

For a group $G$ and $P$ a set of prime numbers, $G$ is \emph{$P$-divisible} if for every element $g\in G$ and every $n$ whose prime divisors are contained in $P$, there exists a unique $h\in G$ such that $h^n=g$. This unique element is denoted $g^{\frac{1}{n}}$. If $G$ is an abelian group and its group operation is denoted by $+$, then we use the notation $\frac{1}{n}g$.

    A \emph{filtered group} is a group $G$ with a filtration of normal subgroups 
    $$G=G_0\supseteq G_1\supseteq G_2\supseteq ...$$ 
    such that $[G_i,G_j]\subseteq G_{i+j}$. Any subgroup $H$ of a filtered group $G$ has a natural filtration given by $H_i=H\cap G_i$. We endow a filtered group with the topology with base $\{gG_i\mid g\in G, i\geq 1\}$. A homomorphism of filtered groups $f:G\to H$ is a group homomorphism such that $f(G_i)\subseteq H_i$. A filtered group $G$ is \emph{complete} if it is Hausdorff and complete as topological space. A \emph{Lazard group} is a complete filtered group $G$ such that $G_0=G_1$ and $G_i$ is $P_i$-divisible for every $i\geq 1$, where $P_i$ is the set of all primes less than or equal to $i$.
    \begin{example}
        On any group $G$ one obtains a natural filtration by setting $G_0=G_1=G$ and $G_{i+1}=[G,G_{i}]$ for $i\geq 1$. If there exists some $k$ such that $G_{k+1}=\{1\}$ (meaning that $G$ is nilpotent) then the induced topology is discrete hence complete. If we moreover assume that $G$ is a $p$-group for some prime $k<p$, then $G$ is easily seen to be a Lazard group since every $p$-group is $P_i$-divisible for $i<p$. However, note that any $P_p$-divisible $p$-group is necessarily trivial.
    \end{example}
    
Similarly, a \emph{filtered Lie ring} is a Lie ring $\fa$ with a filtration of ideals 
$$\fa=\fa_0\supseteq \fa_1\supseteq \fa_2\supseteq ...$$
such that $[\fa_i,\fa_j]\subseteq \fa_{i+j}$ for all $i,j\geq 1$. A homomorphism of filtered Lie rings $f:\fa\to\fb$ is a Lie ring homomorphism such that $f(\fa_i)\subseteq \fb_i$. We endow a filtered group with the topology with base $\{x+\fa_i\mid x\in \fa,i\geq 1\}$. A filtered Lie ring is \emph{complete} if its underlying additive group (with the same filtration) is complete. A \emph{Lazard Lie ring} is a Lie ring whose underlying additive group is Lazard.
\begin{example}
        On any Lie ring $\fa$ one obtains a natural filtration by setting $\fa_0=\fa_1=\fa$ and $\fa_{i+1}=[\fa,\fa_{i}]$ for $i\geq 1$. If there exists some $k$ such that $\fa_{k+1}=\{1\}$ (meaning that $\fa$ is nilpotent) then the induced topology is discrete hence complete. If $(\fa,+)$ is a $p$-group for $k<p$ then $\fa$ is Lazard.
    \end{example}
The benefit of working in a complete Lie ring $\fa$ is that the Baker-Campbell-Hausdorff (BCH) formula can be evaluated in any two elements. Recall that this formula originates in Lie theory and its first terms are given by
\begin{equation}\label{eq: BCH}
    \BCH(x,y)=x+y+\frac{1}{2}[x,y]+\frac{1}{12}([x,[x,y]]+[y,[y,x]])-\frac{1}{24}[y,[x,[x,y]]]+\dots,
\end{equation}
where higher terms are of order at least 5. The prime divisors of the denominator of coefficients of order $n$ only involve primes $\leq n$, which indeed ensures that these are well-defined in $\fa$. The element $\BCH(x,y)$, for $x,y\in \fa$, is then interpreted as $\lim_{k\to \infty} \BCH_k(x,y)$ where $\BCH_k$ is obtained by truncating the formula at order $k$. Note in particular that if $x,y\in \fa$ commute, then $\BCH(x,y)=x+y$. In particular, $\BCH(x,-x)=0$ and $\BCH(x,x)=2x$.

Similarly, in a Lazard group $A$, one can evaluate formal infinite products of commutators as long as their order goes to infinity and as long as the fractional powers involve only primes $\leq n$ for commutators of order $n$. Two examples of such formulae are the first and second inverse BCH formula (see also \cite{Cicalo2012AnCorrespondence}, \cite{Khukhro1998}), whose first terms are
\begin{equation}\label{eq: inverse BCH 1}
    P(g,h)=gh[g,h]^{-\frac{1}{2}}[g,[g,h]]^{\frac{1}{12}}[g,[g,[g,h]]]^{-\frac{1}{24}}[h,[h,[g,h]]]^{\frac{1}{24}}\cdots,
\end{equation}
    and
\begin{equation}\label{eq: inverse BCH 2}
    Q(g,h)=[g,h][g,[g,h]]^{\frac{1}{2}}[h,[g,h]]^{\frac{1}{2}}[g,[g,[g,h]]]^{\frac{1}{3}}[h,[g,[g,h]]]^{\frac{1}{4}}[h,[h,[g,h]]]^{\frac{1}{3}}\cdots,
\end{equation}
where in this case $[g,h]$ is the group theoretic commutator $ghg^{-1}h^{-1}$ and further factors involve iterated brackets of order at least 5.
\begin{theorem}\label{theorem: Lazard}
    Let $\fa$ be a Lazard Lie ring, then $\Laz(\fa):=(\fa,\BCH)$ is a Lazard group with respect to the same filtration. If $f:\fa\to \fb$ is a homomorphism of Lazard Lie rings, then $f:\Laz(\fa)\to \Laz(\fb)$ is a homomorphism of Lazard groups.

    Conversely, let $A$ be a Lazard group. Then $\Laz^{-1}(A):=(A,+,[-,-])$, where $a+b:=P(a,b)$ and $[a,b]:=Q(a,b)$, is a Lazard Lie ring with respect to the same filtration. This construction is also functorial. $\Laz$ and $\Laz^{-1}$ are inverse constructions.
\end{theorem}
\begin{corollary}
    Let $\fa$ be a Lazard Lie ring, then the Lazard sub Lie rings of $\fa$ and Lazard subgroups of $\Laz(\fa)$ coincide. Also, Lazard ideals of $\fa$ and Lazard normal subgroups of $\Laz(\fa)$ coincide.
\end{corollary}
A \emph{filtered ring} $R$ is a ring together with a descending filtration of ideals 
$$A_0\supseteq A_1\supseteq A_2\supseteq \dots$$
such that $A_iA_j\subseteq A_{i+j}$. A filtered ring $R$ is \emph{Lazard} if its underlying additive group is Lazard. If on a Lazard ring we consider the commutator bracket, then we obtain a Lazard Lie ring. Note that a Lazard ring is necessarily non-unital, but one can adjoin a unit to obtain the Dorroh extension $R'=\Z\oplus R$ with filtration $R'_0=R'$, $R'_i=\{0\}\oplus R_i$ for $i\geq 1$. Then the exponential $$\exp(a)=\sum_{k=0}^\infty \frac{1}{k!}a^k=\lim_{n\to \infty}\sum_{k=0}^n\frac{1}{k!}a^k\in 1+R$$ is well-defined for $a\in R$. Similarly, for any $a\in 1+R$ the logarithm $$\log(a)=\sum_{k=1}^\infty (-1)^{k+1}\frac{1}{k} (a-1)=\lim_{n\to \infty}\sum_{k=1}^n (-1)^{k+1}\frac{1}{k} (a-1)\in R$$ 
is a well-defined element and $\exp$ and $\log$ are mutually inverse. Then the Baker-Campbell-Hausdorff formula is obtained as $\BCH(a,b)=\log(\exp(a)\exp(b))$. Similarly, $P(g,h)=\exp(\log(g)+\log(h))$ and $Q(g,h)=\exp(\log(g)\log(h)-\log(h)\log(g))$ where we evaluate $P$ and $Q$ in the Lazard group $1+R$. Explicitly we have the following result.
\begin{proposition}\label{prop: Lazard in ring}
    Let $\fa$ be a Lazard Lie ring embedding in a Lazard ring $R$. Then $\exp(\fa)$ is a Lazard group for the filtration $\exp(\fa)_i=\exp(\fa_i)$ and $\exp:\Laz(\fa)\to \exp(\fa)$ is an isomorphism of Lazard groups.
\end{proposition}
Let $\fa$ be a filtered Lie ring. We define $\End(\fa)$ as the ring of endomorphisms of the underlying additive group $(\fa,+)$. We have a natural filtration on $\End(\fa)$ given by 
    $$\End(\fa)_j=\{f\mid f(\fa_i)\subseteq \fa_{i+j}\},$$
    which makes $\End(\fa)$ into a filtered ring. If $\fa$ is complete, then also $\End(\fa)$ is easily seen to be complete. Let $\End(\fa)^+$ denote the filtered subring of $\End(\fa)$ consisting of all elements in $\End(\fa)_1$. If $\fa$ is Lazard, then it follows directly that $\End(\fa)^+$ is Lazard. 
    \begin{remark}
        Note that although the difference between $\End(\fa)^+$ and $\End(\fa)_1$ is subtle, we consider $\End(\fa)_1$ simply as an ideal of $\End(\fa)$ while $\End(\fa)^+$ is seen as a filtered ring on its own. Throughout this paper we use a similar notation for any filtered ring, group or Lie ring to indicate the filtered substructure consisting of all elements contained in the second term in the filtration.
    \end{remark}
    For $\fa$ a filtered Lie ring and $a\in \fa$, we denote by $\ad_a$ the adjoint map $b\mapsto [a,b]$. Clearly, $\ad:\fa\to \End(\fa):a\mapsto \ad_a$ is a homomorphism of filtered Lie rings, where on $\End(\fa)$ we consider the commutator bracket. 
\begin{lemma}\label{lem: conjugation BCH}
    Let $\fa$ be a Lazard Lie ring, then $\BCH(a,\BCH(b,-a))=\exp(\ad_a)(b)$ for all $a,b\in \fg$.
\end{lemma}
At last, we fix some notation and recall some important substructures of post-Lie rings and skew braces.

Let $(\fa,\tr)$ be a post-Lie ring and $a\in \fa$. We obtain a new Lie bracket on the abelian group $\fa$ by defining $\{a,b\}:=[a,b]+a\tr b-b\tr a$. This Lie ring is denoted by $\fa^\circ$. We will denote left multiplication by $a$ as $\cL_a$, so $\cL_a(b)=a\tr b$. Note that \eqref{eq: postLie 2} states precisely that $\fa^\circ\to \End(\fa):a\mapsto \cL_a$ is a Lie ring homomorphism.
\begin{definition}
    Let $(\fa,\tr)$ be a post-Lie ring. A sub Lie ring $I$ of $(\mathfrak{a},\tr)$ is 
    \begin{itemize}
        \item a \emph{sub post-Lie ring} if $I$ is closed under the operation $\tr$,
        \item a \emph{left ideal} if $x\tr y\in I$ for all $x\in \fa$, $y\in I$,
        \item a \emph{strong left ideal} if it is a left ideal and it is an ideal of $\fa$,
        \item an \emph{ideal} if it is a strong left ideal and moreover $I$ is an ideal of $\fa^\circ$.
    \end{itemize}
\end{definition}
\begin{example}
    Let $(\fa,\tr)$ be a post-Lie ring, then the \emph{fix} $$\Fix(\fa):=\{a\in \fa\mid b\tr a=0 \text{ for all }b\in \fa\},$$ is a left ideal. The \emph{socle} 
    $$\Soc(\fa):=\{a\in \fa\mid a\tr b=[a,b]=0 \text{ for all }b\in \fa\},$$
    and \emph{annihilator}
    $$~~\Ann(\fa):=\{a\in \fa\mid a\tr b=b\tr a =[a,b]=0\text{ for all $b\in \fa$}\},$$
    are ideals. For the latter statement, note that for $a\in \Soc(\fa)$ and $b,c\in \fa$, it follows from \eqref{eq: postLie 1} that $0=b\tr [a,c]=[b\tr a,c]$ and from \eqref{eq: postLie 2} that $0=[a,b]\tr c=-(b\tr a)\tr c$ hence $b\tr a\in \Soc(\fa)$. As $\{b,a\}=b\tr a$, $\Soc(\fa)$ and $\Ann(\fa)$ are ideals.
\end{example}
Let $(A,\cdot,\circ)$ be a skew brace. Then we define $\lambda_a(b):=a^{-1}\cdot (a\circ b)$ and in this way we obtain an action of $(A,\circ)$ on $(A,\cdot)$ by automorphisms 
$$\lambda:(A,\circ)\to \Aut(A,\cdot):a\mapsto \lambda_a.$$
Note that the operation $a*b:=a^{-1}\cdot (a\circ b)\cdot b^{-1}$ can also be expressed as $\lambda_a(b)\cdot b^{-1}$.
\begin{definition}
    Let $(A,\cdot,\circ)$ be a skew brace. A subgroup $I$ of $(A,\cdot)$ is 
    \begin{itemize}  
        \item a \emph{sub skew brace} if $I$ is a subgroup of $(A,\circ)$,
        \item a \emph{left ideal} if $\lambda_a(b)\in I$ for all $a\in A$, $b\in I$,
        \item a \emph{strong left ideal} if $I$ is a left ideal and also a normal subgroup of $(A,\cdot)$,
        \item an \emph{ideal} if $I$ is a strong left ideal and also a normal subgroup of $(A,\circ)$.
    \end{itemize}
\end{definition}
\begin{example}
    Let $(A,\cdot,\circ)$ be a skew brace, then the \emph{fix} $$\Fix(A):=\{a\in A\mid b\circ a=b+a\text{ for all $b\in A$}\},$$
    is a left ideal of $A$. The \emph{socle}
    $$\Soc(A):=\{a\in A\mid a\circ b=a+b=b+a\text{ for all $b\in A$}\},$$
    and the \emph{annihilator }
    $$\Ann(A):=\{a\in A\mid b\circ a=a\circ b=a+b=b+a\text{ for all $b\in A$}\},$$
    are ideals of $A$.
\end{example}

\section{Filtered post-Lie rings}\label{section: post lie}
\begin{definition}
    A \emph{filtered post-Lie ring} is a post-Lie ring $(\fa,\tr)$ together with a filtration $\fa_i$, $i\geq 0$, of the Lie ring $\fa$ where each $\fa_i$ is a strong left ideal.
\end{definition}
    
   Let $\fa$ be a filtered Lie ring. We define $\Der(\fa)$ as the set of all derivations of $\fa$ contained in $\End(\fa)$, these form a sub Lie ring of $\End(\fa)$ under the usual commutator Lie bracket. Thus they have a natural filtration $\Der(\fa)_i=\End(\fa)_i\cap \Der(\fa)$.
Let $\delta:\fg\to \Der(\fa)$ be a homomorphism of filtered Lie rings. We define the semi-direct sum $\fa\oplus_\delta \fg$ as the direct sum of the additive groups together with the bracket
$$[(a,x),(b,y)]=([a,b]+\delta_x(b)-\delta_y(a),[x,y]).$$
It is well-known that this makes $\fa\oplus_\delta \fg$ into a Lie ring. The proof of the following statement is straightforward; the fact that this is indeed a filtration follows directly from the Jacobi identity.
\begin{lemma}
    Let $\fa$, $\fg$ be filtered Lie rings and $\delta:\fg\to \Der(\fa)$ a filtered Lie ring homomorphism. Then $\fa\oplus_\delta \fg$ is a filtered Lie ring for the filtration $(\fa\oplus_\delta \fg)_i=\fa_i\oplus_\delta \fg_i$. If moreover $\fa$ and $\fg$ are Lazard rings then $\fa\oplus_\delta \fg$ is a Lazard ring.
\end{lemma}
In particular, we will be interested in the semidirect sum $\aff(\fa):=\fa\oplus_\delta \Der(\fa)$ where $\delta=\id_{\Der(\fa)}$ and its sub Lie ring $\aff(\fa)^+=\fa^+\oplus \Der(\fa)^+$. If $\fa$ is Lazard, then so is $\Der(\fa)^+$ since it is closed in $\End(\fa)$, and therefore also $\aff(\fa)^+$ is Lazard.

\begin{definition}
    Let $\fa$ be a Lie ring, a sub Lie ring $\fg$ of $\mathfrak{aff
}(\fa)$ is \emph{$t$-surjective}, respectively \emph{$t$-bijective} if the projection $\pr_\fa:\aff(\fa)\to \fa$ to $\pr_\fa:\fg\to \fa$ restricted to $\fg$ yields a surjective, respectively bijective, map from $\fg$ to $\fa$.
\end{definition}

\begin{proposition}\label{prop: correspondence postLie subLie}
    Let $\fa$ be a filtered Lie ring. Then there is a bijective correspondence between operations $\tr:\fa\times \fa\to \fa$ making $(\mathfrak{a},\tr)$ into a filtered post-Lie ring and $t$-bijective sub Lie rings of $\mathfrak{aff}(\fa)$. 
\end{proposition}
\begin{proof}
    This proof is a slight variation of the proof of \cite[Proposition 2.12]{Burde2012AffineStructures}. We therefore only give a sketch of the correspondence. 
    
    If $(\mathfrak{a},\tr)$ is a filtered post-Lie ring then $\{(a,\cL_a)\mid a\in \fa\}$ is its associated $t$-bijective sub Lie ring of $\aff(\fa)$, where $\cL_a(b)=a\tr b$. Since the filtration on $\fa$ consists of left ideals we find indeed that the maps $\cL_a$ are contained in $\End(\fa)$.

    Conversely, if we are given a $t$-bijective sub Lie ring $\fg$ of $\aff(\fa)$ then we define $a\tr b$ as $\cL_a(b)$ where $\cL_a$ is the unique endomorphism of $\fa$ satisfying $(a,\cL_a)\in \fg$. In this way $(\fa,\tr)$ is a filtered post-Lie ring. Since the maps $\cL_a$ respect the filtration on $\fa$ we find that $\fa_i$ is a left ideal for each $i\geq 0$.
\end{proof}
 Recall that for every post-Lie ring $(\fa,\tr)$ we can define the Lie ring $\fa^\circ$ with the same underlying additive structure as $\fa$ and Lie bracket given by \eqref{eq: a^circ}.
\begin{corollary}\label{cor: adjoint lie structure}
    Let $(\fa,\tr)$ be a filtered post-Lie ring. Then $\fa^\circ$ is a filtered Lie ring for the filtration $\fa^\circ_i=\{a\in \fa_i\mid 
    \mathcal{L}_a\in \Der(\fa)_i\}$.
\end{corollary}
\begin{proof}
     Let $\fg$ be the associated $t$-bijective sub Lie ring of $\aff(\fa)$ as in \cref{prop: correspondence postLie subLie}. Then it is easily verified that the bijection $\pr_\fa:\fg\to \fa^\circ$ is a Lie ring homomorphism. By transferring the filtration from $\fg$ onto $\fa^\circ$ the statement follows.
\end{proof}
\begin{definition}
    Let $(\fa,\tr)$ be a filtered post-Lie ring, then $(\fa,\tr)$ is \emph{Lazard} if both $\fa$ and $\fa^\circ$ are Lazard Lie rings.
\end{definition}
We now construct a natural minimal filtration on a post-Lie ring  $(\fa,\tr)$. Inductively define $L^1(\fa)$ as $\fa$ and $L^{i+1}(\fa)$ as the subgroup of $(\fa,+)$ generated by all elements of the form $a\tr b$ and $[a,b]$ where $a\in A$, $b\in L^{i}(A)$. One easily verifies that $L^i(\fa)$ gives a descending filtration of strong left ideals $(\fa,\tr)$. 
\begin{lemma}\label{lem: canonical filtration post-Lie}
    Let $(\fa,\tr)$ be a post-Lie ring, then setting $\fa_0=\fa$ and $\fa_i=L^i(\fa)$ for $i\geq 1$ makes $(\fa_i,\tr)$ into a filtered post-Lie ring.
\end{lemma}
\begin{proof}
    It is straightforward to see that the $\fa_i$ form a descending chain of strong left ideals. It remains to show that $[L^i(\fa),L^j(\fa)]\subseteq L^{i+j}(\fa)$ for all $i,j\geq 1$. Let $\fh=\{\mathcal{L}_a\mid a\in \fa\}\leq \Der(\fa)$ and consider the semi-direct sum $\mathfrak{t}=\fa\oplus_\delta \fh$ where $\delta:\fh\to \Der(\fa)$ is the inclusion map. Then we find $[(a,\cL_b),(c,0)]=([a,c]+b\tr c)$. It follows that $[\mathfrak{t},\mathfrak{t}]=L^2(\fa)\times [\fh,\fh]$. For a Lie ring $\fg$ we introduce the following notation for iterated commutators: $\gamma^1(\fg)=\fg$ and $[\gamma^{i+1}(\fg)=[\fg,\gamma^i(\fg)]$. Then the previous equality can be rewritten as $\gamma^2(\mathfrak{t})=L^2(\fa)\times \gamma^2(\fh)$ and by induction one finds $\gamma^i(\mathfrak{t})=L^i(\fa)\times \gamma^i(\fh)$. It is well-known that $[\gamma^i(\mathfrak{t}),\gamma^j(\mathfrak{t})]\subseteq \gamma^{i+j}(\mathfrak{t})$, so in particular $$[L^i(\fa),L^j(\fa)]\times \{0\}=[L^i(\fa)\times \{0\},L^j(\fa)\times \{0\}]\subseteq L^{i+j}(\fa)\times \gamma^{i+j}(\fh),$$
    thus $[L^i(\fa),L^j(\fa)]\subseteq L^{i+j}(\fa)$.
\end{proof}

The above constructed filtration naturally yields a notion of nilpotency, which we coin $L$-nilpotency. Note that for pre-Lie rings this is the same as left nilpotency.
    
\begin{definition}
    A post-Lie ring $(\fa,\tr)$ is \emph{$L$-nilpotent} if $L^{k+1}(\fa)=0$ for some $k\geq0$. The minimal such $k$, if it exists, is the \emph{$L$-nilpotency class} of $(\fa,\tr)$. 
\end{definition}
\begin{lemma}\label{lem: left nilpotent nilpotent type Lie}
    Let $(\fa,\tr)$ be a post-Lie ring. Then $(\fa,\tr)$ is $L$-nilpotent if and only if $(\fa,\tr)$ is left nilpotent and $\fa$ is a nilpotent Lie ring.
\end{lemma}
\begin{proof}  
    From the definition of $L^i(\fa)$ it is clear that if $(\fa,\tr)$ is $L$-nilpotent then it is left nilpotent and $\fa$ is a nilpotent Lie ring. 
    
    Conversely, assume that $(\fa,\tr)$ is left nilpotent and $\fa$ is a nilpotent Lie ring. Let $\fa=\gamma^1(\fa)\supseteq ...\supseteq \gamma^d(\fa)=\{0\}$ denote the lower central series of $\fa$ and let $c$ be such that $\fa^{c+1}=\{0\}$ (recall that this is the left derived series defined in the introduction). Assume $L^i(\fa)\subseteq \gamma^k(\fa)$ for some $i,k\geq 1$, then $[L^i(\fa),\fa] \subseteq \gamma^{k+1}(\fa)$ hence $L^{i+1}(\fa)\subseteq (\fa\tr L^i(\fa))+ \gamma^{k+1}(\fa)$. In particular, $$L^{i+c}(\fa)\subseteq \fa^{c+1}+ \gamma^{k+1}(\fa)=\gamma^{k+1}(\fa).$$
    We now conclude by induction that $L^{kc+1}(\fa)\subseteq\gamma^{k+1}(\fa)=1$.
\end{proof}

\begin{theorem}\label{theorem: nilpotency class multiplicative lie ring}
    Let $(\fa,\tr)$ be an $L$-nilpotent post-Lie ring of class $k$. Then $\fa^\circ$ is nilpotent of class at most $k$ and the $k$-th term in the lower central series of $\fa^\circ$ is contained in $\Ann(\fa)$.
\end{theorem}
\begin{proof}
    Let $(\fa,\tr)$ be $L$-nilpotent of class $k$ and consider it as a filtered post-Lie ring with its canonical filtration. Then $\fa^\circ_{k+1}\subseteq \fa_{k+1}=0$, so in particular $\fa^\circ$ is nilpotent of class at most $k$. Also, since $\fa^\circ_k=\fa_k\cap \{a\in \fa\mid \mathcal{L}_a\in \End(\fa)_{k}\}$ but $\End(\fa)_{k}=0$ and clearly $\fa_k\in Z(\fa)\cap \Fix(\fa)$, the statement follows.
\end{proof}
We end this section by relating $L$-nilpotency to the notion of transitivity (also called completeness) for pre-Lie algebras, as defined in \cite{Kim1986CompleteGroups}. Let $\F$ be a field of characteristic $0$ and $(\fa,\tr)$ be a finite dimensional pre-Lie $\F$-algebra. Recall that $(\fa,\tr)$ is transitive if for every $a\in \fa$, the map $$r_a:\fa\to \fa: b\mapsto b+a\tr b,$$ is a bijection. However, if $\fa^\circ$ is nilpotent, then this is equivalent to the left multiplication maps $\mathcal{L}_a$ being nilpotent for all $a\in \fa$ \cite[Proposition 2]{Segal1992TheAlgebras}. By Engel's theorem, this is equivalent with the existence of some $n$ such that $\mathcal{L}_{a_1}\cdots \mathcal{L}_{a_n}=0$ for all $a_1,\dots,a_n\in \fa$. It follows that transitivity coincides with $L$-nilpotency, or equivalently with the fact that its associated filtered pre-Lie ring is Lazard. Indeed, the divisibility condition is trivially satisfied, as $\fa_k=0$ for some $k\geq 0$, also $\fa^\circ_k=0$ and thus also the completeness follows for both $\fa$ and $\fa^\circ$. Note that the condition that all $\mathcal{L}_a$ are nilpotent also appears in \cite[Theorem 3.1]{Burde2009AffineGroups}. 

\section{Filtered skew braces}\label{section: skew}

\begin{definition}\label{def: skew braces}
    A \emph{filtered skew brace} is a skew brace $(A,\cdot,\circ)$ together with a descending chain of strong left ideals $A_i$ such that it is a filtration for $(A,\cdot)$
\end{definition}

\begin{definition}
    Let $G$ be a filtered group. By $\Aut(G)$ we denote the filtered group automorphisms of $G$ and we define a filtration by $$\Aut(G)_i=\{f\mid f(g)g^{-1}\in G_{i+j}, \text{ for all }g\in G_j\}.$$
\end{definition}
\begin{lemma}\label{lem: Aut(G) is filtered}
    Let $G$ be a filtered group, then $\Aut(G)$ is a filtered group.
\end{lemma}
\begin{proof}
    Consider the semidirect product $T=G\rtimes \Aut(G)$. We identify $G$ and $\Aut(G)$ with the subgroups $G\times \{1\}$ and $\{1\}\times \Aut(G)$. Note that the subgroups $G_i$ are normal in $T$. By assumption, $[\Aut(G)_i,[\Aut(G)_j,G_k]]\subseteq G_{i+j+k}$, so by the Hall-Witt identity we find that $[G_k,[\Aut(G)_i,\Aut(G)_j]]\subseteq G_{i+j+k}$ and thus we conclude that $[\Aut(G)_i,\Aut(G)_j]\subseteq \Aut(G)_{i+j}$. 
\end{proof}
\begin{lemma}
    Let $A,G$ be filtered groups and $\lambda:G\to Aut(A)$ a homomorphism of filtered groups. Then the semidirect product $A\rtimes_\lambda G$ is a filtered group for the filtration $(A\rtimes_\lambda G)_i=A_i\rtimes_\lambda G_i$ is a filtered group.
\end{lemma}
\begin{proof}
    Since we are working in a semi-direct product we find that $$[A_i\rtimes_\lambda G_i,A_j\rtimes_\lambda G_j]=[A_i,A_j][A_i,G_j][A_j,G_i][G_i,G_j],$$
    where we identitify $A$ and $G$ with the corresponding subgroup of $A\rtimes G$. Clearly $[A_i,A_j]\subseteq A_{i+j}$ and $[G_i,G_j]\subseteq G_{i+j}$. Since $\lambda$ is a homomorphism of filtered group, also $[A_i,G_j]\subseteq A_{i+j}$ and $[A_j,G_i]\subseteq A_{i+j}$ which concludes the proof.
\end{proof}
Let $A$ be a filtered group. We will particularly be interested in the \emph{holomorph} of $A$, which by definition is the semi-direct product $\Hol(A):=A\rtimes \Aut(A)$.
\begin{definition}
    Let $A$ be a filtered group. A subgroup $G\leq \Hol(A)$ is \emph{transitive}, respectively \emph{regular}, if $\pr_G:H\to G$ is surjective, respectively bijective.
\end{definition}
    Note that if we consider the natural action of $\Hol(G)$ on $G$, i.e. $(g,\lambda)\cdot h=g\lambda(h)$, then transitive subgroups of $\Hol(G)$ are precisely the ones such that the associated action is transitive. Regular subgroups are precisely the ones that induces a regular (or also called simply transitive) action.
\begin{proposition}\label{prop: correspondence skew braces subgroups}
    Let $(A,\cdot)$ be a filtered group. There is a bijective correspondence between group structures $(A,\circ)$ making $(A,\cdot,\circ)$ into a filtered skew brace and regular subgroups of $\Hol(A,\cdot)$.
\end{proposition}
\begin{proof}
    The proof is only a slight variation of the proof of
    \cite[Theorem 4.2]{guarnieri2017skew}, therefore we only give a sketch. Let $(A,\circ)$ be a group operation such that $(A,\cdot,\circ)$ is a filtered skew brace. Define $G=\{(a,\lambda_a)\mid a\in A\}\leq \Hol(A)$ then $G$ is a regular subgroup of $\Hol(A)$. Note that $A_i$ are strong left ideals so indeed $\lambda_a$ respects this filtration for all $a\in A$.
    
    Conversely, if $G$ is a regular subgroup of $\Hol(A)$, we transfer the group structure of $G$ onto $A$ with the help of the bijection $\pr_A:G\to A$ (or equivalently $a\circ b=a+\lambda_a(b)$ with $\lambda_a$ the unique map such that $(a,\lambda_a)\in G$) and we denote this new group operation by $\circ$. Then $(A,\cdot,\circ)$ becomes a skew brace and by assumption the maps $\lambda_a$ preserve the filtration $A_i$, hence the $A_i$ are strong left ideals.
\end{proof}
\begin{corollary}\label{cor: filtration on adjoint group}
    Let $(A,\cdot,\circ)$ be a filtered skew brace. Then $(A,\circ)$ with the filtration $(A,\circ)_i=A_i\cap \{a\in A\mid \lambda_a\in \Aut(A)_i\}$ is a filtered group.
\end{corollary}
\begin{proof}
    Let $(A,\cdot,\circ)$ be a filtered skew brace and let $G$ be the associated regular subgroup of $\Hol(A,\cdot)$ as in \cref{prop: correspondence skew braces subgroups}. Then $G_i=G\cap \Hol(A,\cdot)_i$ makes $G$ into a filtered group. Transferring this filtration of $G$ onto $(A,\circ)$ through the isomorphism $(A,\circ)\to G:a\mapsto (a,\lambda_a)$ now yields the statement.
\end{proof}
The notion of a Lazard skew brace is now the natural one.
\begin{definition}
    Let $(A,\cdot,\circ)$ be a filtered skew brace, then $(A,\cdot,\circ)$ is \emph{Lazard} if both $(A,\cdot)$ and $(A,\circ)$ are Lazard groups.
\end{definition}
We now introduce a natural minimal filtration on skew braces. For $(A,\cdot,\circ)$ a skew brace, define inductively $L^1(A)=A$ and $L^{i+1}(A)$ the subgroup of $(A,\cdot)$ generated by all elements of the form $a*b:=a^{-1}\cdot (a\circ b)\cdot b^{-1}$ and $a\cdot b\cdot a^{-1}\cdot b^{-1}$ for $a\in A$, $b\in L^i(A)$. Note that $L^i(A)$ is the left hand version of the series $R_n(A,A)$ as defined in \cite{acri2020retractability}. By induction, one sees that the $L^i(A)$ form a descending chain of strong left ideals of $A$. Also, note that if $A$ is a brace, then $L^i(A)=A^{i}$, so we recover the left derived series.
\begin{lemma}
    Let $(A,\cdot,\circ)$ be a skew brace. Then setting $A_0=A$ and $A_i=L^i(A)$ for $i\geq 1$ makes $(A,\cdot,\circ)$ into a filtered skew brace.
\end{lemma}
\begin{proof}
    The proof is almost word for word the same as \cref{lem: canonical filtration post-Lie}. The only non-trivial part of the statement is that $[L^i(A),L^j(A)]\subseteq L^{i+j}(A)$ for all $i,j\geq 1$. Let $$H=\{\lambda_a\mid a\in A\}\leq \Aut(A,\cdot),$$ and consider the semi-direct product $T=(A,\cdot)\rtimes_\beta H$ where $\beta:H\to \Aut(A,\cdot)$ is the inclusion map. Then we find $[(a,\lambda_b),(c,1)]=([a,c]+b*c,1)$. It follows that $[T,T]=L^2(A)\times [H,H]$. For a group $G$ we define its lower central series as $\gamma^1(G)=G$ and $[\gamma^{i+1}(G)=[\fg,\gamma^i(G)]$. Then the previous equality can be rewritten as $\gamma^2(T)=L^2(A)\times \gamma^2(H)$ and by induction one finds $\gamma^i(T)=L^i(A)\times \gamma^i(H)$. It is well-known that $[\gamma^i(T),\gamma^j(T)]\subseteq \gamma^{i+j}(T)$, so in particular $$[L^i(A),L^j(A)]\times \{0\}=[L^i(\fa)\times \{0\},L^j(\fa)\times \{0\}]\subseteq L^{i+j}(\fa)\times \gamma^{i+j}(\fh),$$
    thus $[L^i(\fa),L^j(\fa)]\subseteq L^{i+j}(\fa)$.
\end{proof}

The canonical filtration introduced above leads to a natural notion of $L$-nilpotency of skew braces, which for braces coincides with left nilpotency.
\begin{definition}
    A skew brace $(A,\cdot,\circ)$ is \emph{$L$-nilpotent} if there $L^{k+1}(A)=1$ for some $k\geq 0$. The minimal such $k$, if it exists, is the \emph{$L$-nilpotency class} of $A$.
\end{definition}

The proof of the following two is completely analogous to those of \cref{lem: left nilpotent nilpotent type Lie} and \cref{theorem: nilpotency class multiplicative lie ring}.
\begin{lemma}\label{lem: left nilpotent nilpotent type brace}
    Let $(A,\cdot,\circ)$ be a skew brace. Then $A$ is $L$-nilpotent if and only if $A$ is left nilpotent and $(A,\cdot)$ is a nilpotent group.
\end{lemma}

\begin{theorem}\label{theorem: nilpotency class multiplicative skew brace}
    Let $(A,\cdot,\circ)$ be an $L$-nilpotent skew brace of class $k$. Then $(A,\circ)$ is nilpotent of class at most $k$ and the $k$-th term in the lower central series of $(A,\circ)$ is contained in $\Ann(A)$.
\end{theorem}

\section{A correspondence between post-Lie rings and skew braces}\label{section: lazard generalized}

Let $\fa$ be a filtered Lie ring. We define $\Aut(\fa)$ as all filtered Lie ring automorphisms of $\fa$ and $\Aut(\fa)_i=\Aut(\fa)\cap \End(\fa)_i$ for $i\geq 0$.

\begin{lemma}\label{lem: lazard derivations and auto}
    Let $\fa$ be a Lazard Lie ring. Then $\Aut(\fa)_i=\exp(\Der(\fa)_i)$ for all $i\geq 1$. In particular, $\exp:\Laz(\Der(\fa)^+)\to \Aut(\fa)^+$ is a group isomorphism and thus $\Aut(\fa)^+$ is a Lazard group.
\end{lemma}
\begin{proof}
      Clearly $\exp(\End(\fa)_i)= 1+\End(\fa)_i$ for $i\geq 1$. Let $\delta\in \Der(\fa)_1$ and $x,y\in \fa$. By induction one finds that $\delta^n([a,b])=\sum_{j=0}^n\binom{n}{j}[\delta^j(a),\delta^{n-j}(b)]$ for $n\geq 0$, hence
    \begin{align*}
        \exp(\delta)([a,b])&=\sum_{n=0}^\infty\sum_{j=0}^i\frac{1}{n!}\binom{n}{j}[\delta^j(a),\delta^{n-j}(b)]=\sum_{i=0}^\infty\sum_{j=0}^\infty \frac{1}{i!j!}[\delta^j(a),\delta^{i}(b)]\\
        &=[\exp(\delta)(a),\exp(\delta)(b)].
    \end{align*}
    Conversely, let $\delta\in \End(\fa)_1$ such that $\exp(\delta)\in \Aut(\fa)$. We will prove by induction on $k$ that $$\delta([a,b])\in [\delta(a),b]+[a,\delta(b)]+\fa_{i+j+k},$$
    for all $a\in \fa_i$, $b\in \fa_j$ which implies that $\delta([a,b])-[\delta(a),b]-[a,\delta(b)]\in \bigcap_{k=0}^\infty =\{0\}$ and thus $\delta\in \Der(\fa)$. For $k=1$, the statement is trivial. Assume that it holds for some $k\geq 1$. We prove by induction on $n$ that
    \begin{equation}\label{eq: induction n}
        \delta^n([a,b]\in \sum_{j=0}^n\binom{n}{j}[\delta^j(a),\delta^{n-j}(b)]+\fa_{i+j+k+1}
    \end{equation} 
    for $a\in \fa_i$ , $b\in \fa_j$ and $n\geq 2$. By assumption, $$\delta([a,b])\in [\delta(a),b]+[a,\delta(b)]+\fa_{i+j+k},$$ so by applying $\delta$ on both sides we find 
    $$\delta^2([a,b])\in \delta([\delta(a),b])+\delta([a,\delta(b)])+\fa_{i+j+k+1}.$$
    By assumption, 
    \begin{align*}
        \delta([\delta(a),b])\in[\delta^2(a),b]+[\delta(a),\delta(b)]+\fa_{i+j+k+1},\\
        \delta([a,\delta(b)])\in [\delta(a),\delta(b)]+[a,\delta^2(b)]+\fa_{i+j+k+1},
    \end{align*}
    so indeed \eqref{eq: induction n} holds for $n=2$. Similarly it follows that \eqref{eq: induction n} holds for $n\geq 2$. We then use \eqref{eq: induction n} to obtain
    \begin{align*}
        \exp(\delta)([a,b])\in[a,b]+\delta([a,b])+\sum_{n=2}^\infty\sum_{j=0}^n\frac{1}{j!(n-j)!}[\delta^j(a),\delta^{n-j}(b)]+\fa_{i+j+k+1},
    \end{align*}
    while also
    $$[\exp(\delta)(a),\exp(\delta)(b)]=\sum_{i=0}^\infty\sum_{j=0}^\infty \frac{1}{i!j!}[\delta^j(a),\delta^{i}(b)].$$
    Since $\exp(\delta)([a,b])=[\exp(\delta)(a),\exp(\delta)(b)]$ we find $\delta([a,b])\in [\delta(a),b]+[a,\delta(b)]+\fa_{i+j+k+1}$. We conclude that indeed $\exp(\Der(\fa)_i)=\Aut(\fa)_i$. The last part of the statement now follows from \cref{prop: Lazard in ring}.
\end{proof}
\begin{lemma}
    Let $\fa$ be a Lazard Lie ring, then $\Aut(\fa)=\Aut(\Laz(\fa))$ as filtered groups. 
\end{lemma}
\begin{proof}
    We know by \cref{theorem: Lazard} that automorphisms of $\fa$ and $\Laz(\fa)$ coincide. It remains to verify that the filtrations coincide. Let $x,y\in \fa$ such that $x+y\in \fa_i\setminus\fa_{i+1}$, then $$[x,y]+\fa_{i+1}=[x,y]-[x+y,y]+\fa_{i+1}=\fa_{i+1},$$
    hence $x+y+\fa_{i+1}=\BCH(x,y)+\fa_{i+1}$ and thus $\BCH(x,y)\in \fa_{i}\setminus \fa_{i+1}$. We conclude that $x+y\in \fa_i$ if and only if $\BCH(x,y)\in \fa_i$. In particular, for $\Aut(\fa)$ and $x\in \fa_i$ we find $\BCH(\phi(x),-x)\in \fa_{i+j}$ if and only if $\phi(x)-x\in \fa_{i+j}$. We conclude that indeed $\Aut(\fa)_i=\Aut(\Laz(\fa))_i$.
\end{proof}

\begin{proposition}\label{prop: gamma}
    Let $\fa,\fg$ be Lazard Lie rings and $\delta:\fg\to \Der(\fa)$ a homomorphism of filtered Lie rings. Then there is a unique filtered group isomorphism $$\gamma:\Laz(\fa\oplus_\delta \fg)\to \Laz(\fa)\rtimes_\lambda \Laz(\fg),$$ such that $\gamma(a,0)=(a,0)$ and $\gamma(0,x)=(0,x)$, where $\lambda:\Laz(\fg)\to \Aut(\Laz(\fa))$ is given by $\lambda_x=\exp(\delta_x)$.
\end{proposition}

\begin{proof}
    Consider the split short exact sequence of Lazard Lie rings
\[
\begin{tikzcd}
0 \arrow[r] & \fa \arrow[r] & \fa\oplus_\delta \fg \arrow[r, shift right] & \fg \arrow[r] \arrow[l, shift right] & 0 \arrow[l, phantom, bend right].
\end{tikzcd}\]
    The functoriality of $\Laz$ then yields a short exact sequence of Lazard groups
\[
\begin{tikzcd}
1 \arrow[r] & \Laz(\fa) \arrow[r] & \Laz(\fa\oplus_\delta \fg) \arrow[r, shift right] & \Laz(\fg) \arrow[r] \arrow[l, shift right] & 1 \arrow[l, phantom, bend right].
\end{tikzcd}\]
    Together with \cref{lem: conjugation BCH} this proves the statement.
\end{proof}
We continue with the same setting as \cref{prop: gamma}. Since $\fa\times \{0\}$ is a normal subgroup of both $\Laz(\fa\oplus_\delta \fg)$ and $\Laz(\fa)\rtimes_\lambda \Laz(\fg)$ and $\gamma(\fa\times \{0\})=\fa\times \{0\}$, we find that $\gamma$ does not affect the second component.  Define $V:\fa\times \fg\to \fa$ as $\pr_\fa\gamma$, thus $\gamma(a,x)=(V(a,x),x)$. Then in $\Laz(\fa)\rtimes_\lambda \Laz(\fg)$ we find
$$(V(a,x),0)=(V(a,x),x)(0,x)^{-1}=\gamma(a,x)\gamma(0,x)^{-1}=\gamma(\BCH((a,x),(0,-x))),$$
from which we conclude that $V(a,x)=\pr_\fa \BCH((a,x),(0,-x))$. From \eqref{eq: BCH} we find
\begin{equation}\label{eq: formula V}
    V(a,x)=a+\frac{1}{2}\delta_x(a)+\frac{1}{6}\delta^2_x(a)+\frac{1}{12}[a,\delta_x(a)]+\frac{1}{24}(\delta_x([a,\delta_x(a)])+\delta^3_x(a))+...
\end{equation}
    where further terms are of order at least 5.

    Similarly, we define $U:\fa\times \fg\to \fa$ as $U=\pr_\fa\gamma^{-1}$ hence $\gamma^{-1}(a,x)=(U(a,x),x)$. We find $U(a,x)=\pr_\fa(P((a,\lambda),(0,\lambda^{-1})))$ with $P$ as in \eqref{eq: inverse BCH 1}. Explicitly the terms of order less than 5 are given by 
    \begin{align}
        U(a,\lambda)=a&(a\lambda^{-1}(a)^{-1})^{-\frac{1}{2}}(a\lambda(a)a^{-2}\lambda^{-1}(a)a^{-1})^{\frac{1}{12}}\label{eq: formula U}\\
        &(a\lambda(a)\lambda^2(a)\lambda(a)^{-2}a\lambda(a)^{-1}\lambda^{-1}(a)^{-1}a^2\lambda(a^{-1})a^{-1})^{-\frac{1}{24}}\nonumber\\
        &(\lambda^2(a)\lambda(a)^{-1}a\lambda(a)^{-1}a\lambda^{-1}(a)^{-1}a\lambda(a)^{-1})^{\frac{1}{24}}\cdots\nonumber
    \end{align}
In the case that $\fa$ is abelian we have more explicit formulae.
\begin{lemma}\label{lem: W abelian}
     Let $\fa,\fg$ be Lazard Lie rings and $\delta:\fg\to \Der(\fa)$ a homomorphism of filtered Lie rings. Let $V,U:\fa\times \fg\to \fa$ be as defined above. If $\fa$ is abelian, then 
     \begin{align}
        V(a,x)&=\sum_{k=1}^\infty \frac{1}{k!}\delta_x^{k-1}(a),\\
        U(a,x)&=\sum_{k=2}^\infty \frac{1}{k}(-1)^{k+1}(\lambda_x-\id)^{k-1}(a).
     \end{align}
     
\end{lemma}
\begin{proof}
     Consider the abelian group $M=\Z\oplus \fa$ with filtration $M_i=0\oplus \fa_i$ for $i\geq 1$. Let $R=\End(M)^+$ and define 
     $$\rho:\fa\oplus_\delta \fg\to R:(a,x)\mapsto \rho_{(a,x)},$$
     as $\rho_{(a,x)}(n,b)=(0,na+\delta_x(b))$. Then $\rho$ is a Lie ring homomorphism, indeed for $a,b,c\in \fa$, $x,y\in \fg$, $n\in \Z$, we have
     \begin{align*}
         (\rho_{(a,x)}\rho_{(b,y)}-\rho_{(b,y)}\rho_{(a,x)})(n,c)&=\rho_{(a,x)}(0,nb+\delta_y(c))-\rho_{(b,y)}(0,na+\delta_x(c))\\
         &=(0,\delta_x(nb)+\delta_x\delta_y(c)-\delta_y(na)-\delta_y\delta_x(c))\\
         &=(0,n(\delta_x(b)-\delta_y(a))+\delta_{[x,y]}(c))\\
         &=\rho_{(\delta_x(b)-\delta_y(a),[x,y])}\\
         &=\rho_{[(a,x),(b,y)]}.
     \end{align*}
     By \cref{prop: Lazard in ring} we find that $$\kappa:\Laz(\aff(\fa)^+)\to \id_M+R:(a,x)\mapsto \kappa_{(a,x)},$$ with $\kappa_{(a,x)}=\exp(\rho_{(a,x)})$ is a group homomorphism. It also follows directly from the definition of $\kappa$ that $$\kappa_{(a,x)}(1,0)=(1,\sum_{k=1}^\infty \frac{1}{k!}\delta_x^{k-1}(a)),$$ 
     so in particular $\kappa_{(a,0)}(1,0)=(1,a)$ and $\kappa_{(0,x)}(1,0)=(1,0)$.
     We now obtain \eqref{eq: formula V} from
     \begin{align*}
         (1,V(a,\delta))&=\kappa_{(V(a,x),0)}(1,0)\\
         &=\kappa_{\BCH((a,x),(0,-x))}(1,0)\\
         &=\kappa_{(a,x)}\kappa_{(0,-x)}(1,0)\\
         &=(1,\sum_{k=1}^\infty \frac{1}{k!}\delta_x^{k-1}(a)).
     \end{align*}
    Next we prove \eqref{eq: formula U}. Since $\kappa$ is a group homomorphism, so is $f=\kappa\gamma^{-1}$. Thus $\kappa_{(U(a,x),x)}=f(a,x)=f(a,0)f(0,\lambda)=\kappa_{(a,0)}\kappa_{(0,x)}$ and therefore $$(0,U(a,x))=\rho_{(U(a,x),x)}(1,0)=\log(\kappa_{(a,0)}\kappa_{(0,x)})(1,0).$$ 
    We obtain \eqref{eq: formula U} since $(\kappa_{(a,0)}\kappa_{(0,x)}-\id)^{k}(1,0)=(0,(\lambda_x-\id)^{k-1}(a))$ for $k\geq 1$.
\end{proof}

We now consider a particular case of the above. Let $\fa$ be a Lazard Lie ring, then by \cref{prop: gamma} we have a Lazard group isomorphism $\gamma: \Laz(\aff(\fa)^+)\to \Laz(\fa)\rtimes \Laz(\Der(\fa)^+)$. Identifying $\Laz(\Der(\fa)^+)$ with $\Aut(\fa)^+$ as in \cref{lem: lazard derivations and auto} and $\Aut(\fa)^+$ with $\Aut(\Laz(\fa))^+$ we find an isomorphism 
$$\gamma:\Laz(\aff(\fa)^+)\to \Hol(\Laz(\fa))^+,$$
such that $\gamma(a,0)=(a,0)$ and $\gamma(0,\delta)=(0,\exp(\delta))$.

\begin{proposition}\label{prop: correspondence gamma}
    Let $\fa$ be a Lazard Lie ring. Then $\gamma$ yields a bijective correspondence between Lazard sub Lie rings $\fg$ of $\mathfrak{aff}(\fa)^+$ and Lazard subgroups $G$ of $\Hol(\Laz(\fa))^+$. Moreover, $\pr_\fa: \fg\to \fa$ is injective, respectively surjective, if and only if $\pr_\fa:\gamma(\fg)\to \fa$ is injective, respectively surjective.
\end{proposition}
\begin{proof}
    The first part follows directly from \cref{theorem: Lazard} and \cref{prop: gamma}.
    
    Assume that $\fg$ is a Lazard sub Lie ring of $\mathfrak{aff}(\fa)$ such that $\pr_\fa:\fg \to \fa$ is not injective. Then there exists some non-zero element $(0,\delta)\in \fg$. But then clearly $\gamma(0,\delta)=(0,\exp(\delta))\in \gamma(\fg)$ and thus $\pr_\fa:\gamma(\fg)\to A$ is not injective. 

    Conversely, assume that there exist $(a,\lambda),(a,\lambda')\in \gamma(\fg)$ with $\lambda\neq \lambda'$. Then $(a,\lambda)^{-1}(a,\lambda')=(0,\lambda^{-1}\lambda')\in \gamma(\fg)$.
    Hence $\gamma^{-1}(0,\lambda^{-1}\lambda')=(0,\log(\lambda^{-1}\lambda'))\in \fg$ and thus $\pr_\fa:\fg\to \fa$ is not injective.

    Assume that $\pr_\fa: \fg\to \fa$ is surjective. Since $\pr_A$ and $\gamma$ are continuous, the set $\pr_\fa\gamma(\fg)=V(\fg)$ is closed in $\fa$. Assume that there exists some $a\in \fa\setminus V(\fg)$, then there exists a $d$ such that $V(\fg)\cap(a+\fa_d)\neq \varnothing$ but $X\cap(a+\fa_{d+1})= \varnothing$. Let $(b,\delta)\in \fg$ such that $V(b,\delta)\in a+\fa_d$. By assumption there exists some $\delta'\in \Der(\fa)^+$ such that $(-V(b,\delta)+a,\delta')\in \fg$. Since $-V(b,\delta)+a\in \fa_{d}$, we have
    \begin{equation}
    \begin{aligned}
        \label{eq: V equal id mod}
        -V(b,\delta)+a+\fa_{d+1}&=V(-V(b,\delta)+a,\delta')+\fa_{d+1}\\
        &=\exp(\delta)(V(-V(b,\delta)+a,\delta'))+\fa_{d+1}.
    \end{aligned}
    \end{equation}
    Hence  
    \begin{align*}
    V(\BCH((b,\delta),(-V(b,\delta)+a,&\delta')))+\fa_{d+1}\\
    &=\BCH(V(b,\delta),\exp(\delta)V(-V(b,\delta)+a,\delta'))+\fa_{d+1}\\
    &=\BCH(V(b,\delta),-V(b,\delta)+a)+\fa_{d+1}\\
    &=a+\fa_{d+1},
    \end{align*} 
    where the first equality follows from the fact that $\gamma$ is a group homomorphism. We conclude that $V(\fg)\cap (a+\fa_{d+1})\neq \varnothing$ which is a contradiction.

    Conversely, assume that $V(\fg)=\fa$. Clearly $\pr_\fa(\fg)$ is closed in $\fa$ because $\pr_\fa$ is continuous. Assume that there exists some $a\in \fa\setminus \pr_\fa(\fg)$ and let $d$ be such that $\pr_\fa(\fg)\cap(a+\fa_d)\neq \varnothing$ but $\pr_\fa(\fg)\cap (a+\fa_{d+1})=\varnothing$. Choose $(b,\delta)\in \fg$ such that $b\in a+\fa_d$. By assumption there exists some $(b',\delta')\in \fg$ such that $V(b',\delta')=a-b$ and using \eqref{eq: V equal id mod} we find $a-b=V(b',\delta')=b'$. We thus obtain the contradiction
    \[
    \pr_{\fa}((b,\delta)+(b',\delta'))+\fa_{d+1}=a+\fa_{d+1}.\qedhere\]
\end{proof}
\begin{remark}
    Transitive subgroups of the holomorph of a finite group play an essential role in the theory of Hopf--Galois structures, see for example \cite{Childs1989,Martin-Lyons2024SkewBracoids,Stefanello2023OnBraces}. One can prove that for any finite $p$-group $A$ there exists a filtration of $A$ such that $\Hol(A)^+$ is a Sylow $p$-subgroup of the usual holomorph of $A$ (just considered as a group without a filtration). If $|A|=p^k$ for $k<p$, then such a filtration has length at most $p$, so $A$ is a Lazard group and so is any subgroup of $\Hol(A)^+$. Hence in such a case, it follows from \cref{prop: correspondence gamma} that we can classify all regular $p$-groups of the holomorph of $A$ up to conjugation by classifying the $t$-surjective sub Lie rings of $\aff(\Laz^{-1}(A))^+$.
\end{remark}

\begin{proposition}\label{prop: group of flows}
    Let $(\fa,\tr)$ be a Lazard post-Lie ring. Then 
    \[W:\fa\to \fa:a\mapsto V(a,\cL_a)\]
    is a bijective map and $\bS(\fa,\tr)=(\fa,\cdot,\circ)$ with $a\cdot b=\BCH(a,b)$, $a\circ b=a\cdot \exp(\mathcal{L}_{\Omega(a)})(b)$ is a Lazard skew brace, where $\Omega:=W^{-1}$. Moreover, $W:\Laz(\fa^\circ)\to (\fa,\circ)$ is an isomorphism of filtered groups.

    If $\fa$ is abelian, then $$W(a)=\sum_{k=1}^\infty \frac{1}{k!}\cL_a^{k-1}(a).$$
\end{proposition}
\begin{proof}
    Let $(\fa,\tr)$ be a Lazard post-Lie ring and let $\fg$ be the regular Lazard sub Lie ring of $\aff(\fa)^+$ associated to $(\fa,\tr)$ as in \cref{prop: correspondence postLie subLie}. Then $W$ is a bijection since it is the composition of the bijections $\fa\to \fg:a\mapsto (a,\mathcal{L}_a)$ and $V:\fg\to \fa$, recall that the latter is a bijection by \cref{prop: correspondence gamma}. In the case that $\fa$ is abelian, the explicit formula of $W$ follows from \cref{lem: W abelian}.
    
    By \cref{prop: correspondence gamma} we find that $\gamma(\fg)$ is a regular Lazard subgroup of $\Hol(\Laz(\fa))^+$, thus we obtain a Lazard skew brace $(\fa,\cdot,\circ)$ where $(\fa,\cdot)=\Laz(\fa)$. Also $a\circ b=a\cdot \lambda_a(b)$ where $\lambda_a\in \Aut(\fa)$ is uniquely determined by the fact that $(a,\lambda_a)\in \gamma(\fg)$. Since
    $\gamma(\fa)=\{(W(a),\exp(\mathcal{L}_a))\mid a\in \fa\}$
    we find $\lambda_a=\exp(\mathcal{L}_{\Omega(a)})$. 

    Also, $W: \Laz(\fa^\circ)\to (\fa,\circ)$ is a group isomorphism because it is the composition of the group homomorphisms $\Laz(\fa^\circ) \to \Laz(\fg):a\mapsto (a,\mathcal{L}_a)$, $\gamma:\Laz(\fg)\to \gamma(\fg)$ and $\pr_{\fa}:\gamma(\fg)\to (\fa,\circ)$.
\end{proof}
Note that in the case that $\fa$ is abelian, one obtains the group of flows as in \cite{Agrachev1981ChronologicalFields}.
It is less trivial to see how the above construction is related to the formal integration as described by Bai, Guo, Sheng and Tang in \cite{Bai2023Post-groupsEquation}. Let $\F$ be a field of characteristic $0$. Then Bai, Guo, Sheng and Tang proved that to any connected complete post-Lie $\F$-algebra $(\fa,\tr)$ one can associate a skew brace $(\fa,\cdot,\circ)$ where
$$a\cdot b=\BCH(a,b),~~~~a\circ b=a\cdot \exp(\mathcal{L}_{\Omega'(a)})(b).$$
We first note that in our terminology their notion of a connected complete post-Lie $\F$-algebra is a Lazard post-Lie $\F$-algebra $(\fa,\tr)$ such that moreover $\fa_i\tr \fa_j\subseteq \fa_{i+j}$, or equivalently, such that $\fa_i^\circ=\fa_i$ for all $i\geq 0$. Note that under this extra assumption, if $\fa_i=0$ for some $i$, then $\fa_{i-1}$ is contained in the socle hence $\Soc(\fa,\tr)\neq 0$. In \cite[Section 3]{Fried1986DistalityStructures} an example is given of a left nilpotent pre-Lie $\mathbb{R}$-algebra with trivial socle, so it can be given a filtration which makes it a Lazard Lie $\mathbb{R}$-algebra but not a connected complete $\mathbb{R}$-algebra in the sense of Bai, Guo, Sheng and Tang.

We quickly recall their construction: Let $(\fa,\tr)$ be a connected complete post-Lie $\F$-algebra $(\fa,\tr)$ and consider the universal enveloping algebra $\U(\fa)$ together with the filtration where $\U(\fa)_0=\U(\fa)$ and $\U(\fa)_i$ is the subspace generated by $x_1\cdots x_n$ where $x_j\in \fa_{d_j}$ such that $d_1+\dots+d_n=i$. This makes $\U(\fa)$ into a filtered $\F$-algebra. On $\U(\fa)$ we define a new multiplication $\star $ which for $a\in \fa$, $b\in \U(\fa)$ is given by $a\star b=ab+\delta_a(b)$, where $\delta_a$ is the unique derivation of $\U(\fa)$ with $\delta_a(z)=a\tr z$ for $z\in \fa$. In this way, $(\U(\fa),\star )$ becomes a filtered $\F$-algebra (where we keep the same filtration as before). This operation is then extended to the completion $\hU(\fa)$ with respect to the filtration in order to obtain a complete $\F$-algebra $(\hU(\fa),\star )$. For $a\in \fa$, we define $\Omega'(a)=\log_{\star }\exp(a)$, where the exponential map with respect to the usual multiplication in $\hU(\fa)$ is first applied, followed by the logaritmic map in $(\hU(\fa),\star )$. 

Recall that $\fa$ embeds into $\hU(\fa)$. One sees that 
$\rho: \aff(\fa)\to \End(\hU(\fa))$,
given by $\rho_{(a,\delta)}(b)=ab+\delta(b)$ is a homomorphism of filtered Lie $\F$-algebras.
Here we do not distinguish between a derivation $\delta$ of $\fa$ and its unique extension to a derivation of $\hU(\fa)$. Clearly, $a\star b=\rho_{(a,\cL_a)}(b)$ for all $a\in \fa$ and $b\in \hU(\fa)$, thus $\exp_\star (a)=\exp(\rho_{(a,\cL_a)})(1)$. Since derivations map $1$ to $0$, also $$\exp_\star (a)=\exp(\rho_{(a,\cL_a)})\exp(\rho_{(0,-\cL_a)})(1),$$
where $\exp_\star $ denotes the exponential in $(\hU(\fa),\star )$. By \cref{prop: Lazard in ring} and the definition of $W$, we find $$\exp(\rho_{(a,\cL_a)})\exp(\rho_{(0,-\cL_a)})=\exp(\rho_{(W(a),0)}),$$
thus $\exp_\star (a)=\exp(W(a))$. We conclude that $\Omega=\Omega'$.
\begin{proposition}\label{prop: back to Lie}
    Let $(A,\cdot,\circ)$ be a Lazard skew brace. Then 
    $$\Omega:A\mapsto A:a\mapsto U(a,\lambda_a),$$
    is a bijection and $\bL(A,\cdot,\circ)=(\Laz^{-1}(A,\cdot),\tr)$ with 
    $a\tr b=\log(\lambda_{W(a)})(b),$ is a Lazard Lie ring, where $W:=\Omega^{-1}$. Moreover, 
    $\Omega: (A,\circ)\to \fa^\circ$ is an isomorphism of filtered Lie rings, where $(\fa,\tr)=\bL(A,\cdot)$.

    If $(A,\cdot)$ is abelian, then $$\Omega(a)=\sum_{k=2}^\infty \frac{1}{k}(-1)^{k+1}(\lambda_a-\id)^{k-1}(a).$$
\end{proposition}
\begin{proof}
    Let $(A,\cdot,\circ)$ be Lazard skew brace and let $G$ be the associated regular Lazard subgroup of $\Hol(A,\cdot)^+$. Then $\Omega$ is a bijection since it is the composition of the bijections $A\mapsto G:a\mapsto (a,\lambda_a)$ and $U:G\to A$ with $U$. The fact that $U$ is bijective, follows from \cref{prop: correspondence gamma}.

    From \cref{prop: correspondence gamma} it also follows that $\gamma^{-1}(G)$ is a Lazard sub Lie ring of $\aff(\Laz^{-1}(A,\cdot))^+$, thus we obtain a Lazard post-Lie ring $(\Laz^{-1}(A,\cdot),\tr)$. We find $a\tr b=\mathcal{L}_a$ where $\mathcal{L}_a$ is uniquely determined by the fact that $(a,\mathcal{L}_a)\in \gamma^{-1}(G)$. Since $\gamma^{-1}(G)=\{(\Omega(a),\log(\lambda_a))\mid a\in A\}$ we find that $\mathcal{L}_a=\log(\lambda_{\Omega^{-1}(a)}).$

    Also, $\Omega$ is $(A,\circ)\to \bL(A,\cdot,\circ)^\circ$ is an isomorphism of filtered Lie rings since it is the composition of the isomorphisms $(A,\circ)\to G:a\mapsto (a,\lambda_a)$, $\gamma^{_1}:G\to \gamma^{-1}(G)$ and $\pr_{A}:\gamma^{-1}(G)\to \Laz^{-1}(A,\cdot)$.
\end{proof}

\begin{theorem}\label{theorem: final correspondence}
    The constructions in \cref{prop: group of flows} and \cref{prop: back to Lie} are mutually inverse and functorial. The map $\Omega$ associated to a Lazard post-Lie ring $(\fa,\tr)$ coincides with the map $\Omega$ associated to $\bS(\fa,\tr)$.

    Lazard sub post-Lie rings of $(\fa,\tr)$ and Lazard sub skew braces of $\bS(\fa,\tr)$ coincide and moreover, Lazard left ideals, Lazard strong left ideals and Lazard ideals of $(\fa,\tr)$ and $\bS(\fa,\tr)$ coincide.
\end{theorem}
\begin{proof}
    It is straightforward to verify that $\bS$ and $\bL$ are mutually inverse constructions and that their associated maps $\Omega$ coincide. We prove the functoriality of $\bS$, the proof for $\bL$ is analogous. Let $f:(\fa,\tr)\to (\fb,\tr)$ be a homomorphism of Lazard post-Lie rings and set $(\fa,\cdot,\circ)=\bS(\fa,\tr)$ and $(\fb,\cdot,\circ)=\bS(\fb,\tr)$. Let $W_\fa$ and $W_\fb$ denote their respective maps $W$. Then by \cref{theorem: Lazard} $f:(\fa,\cdot)\to (\fb,\cdot)$ is a homomorhism of Lazard groups. As $W$ is an isomorphism between $\Laz(\fa^\circ)$ and $(\fa,\circ)$, it is sufficient to proof $fW_\fa=W_\fb f$ in order to conclude that $f: (\fa,\circ)\to (\fb,\circ)$ is a group homomorphism. This equality follows direct from the fact that $f\mathcal{L}_a=\mathcal{L}_{f(a)}f$ for all $a\in \fa$.

    The equivalence of Lazard post-Lie rings and Lazard sub skew braces is a direct corollary of \cref{theorem: final correspondence}. In particular, note that for $\fb$ a Lazard sub post-Lie ring of $(\fa,\tr)$ we have $W_\fb(b)=W_\fa(b)$. In particular, $W_\fa(\fb)=\fb$.

    Clearly the Lazard sub Lie rings of $\fa$ invariant under $\{\mathcal{L}_x\mid x\in \fa\}$ and the ones invariant under $\{\exp(\mathcal{L}_x)\mid x\in \fa\}$ coincide, which yields the equivalence of Lazard left ideals. Taking into account \cref{theorem: Lazard}, the equivalence of Lazard strong left ideals follows. 

    Let $\fb$ be a Lazard strong left ideal of $(\fa,\tr)$. As $W:\Laz(\fa^\circ)\to (A,\circ)$ is a group isomorphism, we find that $\fb$ is an ideal of $\fa^\circ$ if and only if $W(\fb)$ is a normal subgroup of $(A,\circ)$. Since $W(\fb)=\fb$, the equivalence of Lazard ideals of $(\fa,\tr)$ and $\bS(\fa,\tr)$ follows.
\end{proof}
\begin{remark}
    Recall that a skew brace $(A,+,\circ)$ is \emph{square-free} if $a*a=0$ for all $a\in A$. Similarly, we say that a post-Lie ring $(\fa,\tr)$ is \emph{square-free} if $a\tr a=0$ for all $a\in \fa$. Note that the property of being square-free is preserved by  $\bS$ and $\bL$ and $W=\Omega=\id$ in this case. Therefore these classes behave particularly well under the correspondence and eliminate most of the computations involved.
\end{remark}

    The construction $\bL$ coincides (under the necessary assumptions) with the geometric construction in \cite[Theorem 4.3]{Bai2023Post-groupsEquation} (or in \cite[Section 3]{Burde2009AffineGroups}). We shortly recall this construction. Let $(G,\cdot,\circ)$ be a skew brace such that $(G,\cdot)$ and $(G,\circ)$ are Lie groups and the map $a\mapsto (a,\lambda_a)$ is smooth. One defines a post-Lie structure on $\fg$, the Lie algebra associated to $(G,\cdot)$, by starting with the group homomorphism $\lambda:G\to \Aut(G,\cdot):a\mapsto \lambda_a$, which naturally yields a group homomorphism $\alpha:G\to \Aut(\fg):a\mapsto d\lambda_a$. Subsequently, differentiation yields a map $d\alpha:\fg\to \Der(\fg)$ and $(\fg, \tr')$ is a post-Lie algebra where $a\tr' b=(d\alpha)_a(b)$.

  Let $(\fa,\cdot,\circ)=\bS(\fa,\tr)$ for a finite dimensional Lazard Lie $\mathbb{R}$-algebra $(\fa,\tr)$. Then $(\fa,\cdot)$ is a Lie group and its associated Lie algebra associated is isomorphic to $\fa$ such that under this identification, the exponential map $\fa\to (\fa,\cdot)$ is just the identity. Thus, under this identification of $(\fa,\cdot)$ and $\fa$ we have $\lambda_g=\alpha_g$. Hence, $b\mapsto a\tr'b$ is the differential of the map $\lambda:a\mapsto \exp(\mathcal{L}_{\Omega(a)})(b)$. We first determine $d\Omega$. Since $\Omega W=\id$, it is sufficient to determine $dW$. By \eqref{eq: formula V}, we find $dW=\id$ hence also $d\Omega=\id$. The differential of the map $\exp(\mathcal{L}):\fa\to \Aut(\fa):a\mapsto \exp(\mathcal{L}_a)$ in $a\in \fa$ is
  $$d(\exp(\mathcal{L}))(a)=\left.\frac{d}{dt}\right\rvert_{t=0} \exp(\mathcal{L}_{ta})=\mathcal{L}_a.$$ 
  We conclude that $(d\lambda)_a=\mathcal{L}_a$ and thus $a\tr' b=a\tr b$.

\begin{example}
    Let $A$ be a Lazard ring and let $\mathcal{L}_a$ denote left multiplication by $a$. Then $A^+$ is a Lazard pre-Lie algebra for $a\tr b=ab$ and it is well-known that the associated group of flows is given by $a\circ b=a+ab+b$. We can verify this directly: clearly $W(a)=\exp(a)-1$, hence $\Omega(a)=\log(1+a)$. We now find $$a\circ b=a+\exp(\mathcal{L}_{\log(1+a)})(b)=a+ab+b.$$
    Conversely, starting from the brace $(A,+,\circ)$ we find $\Omega(a)=\log(1+a)$ and thus $W(a)=\exp(a)-1$. we recover the original multiplication as $$a\tr b=\log(\lambda_{W(a)})(b)=\log(\mathcal{L}_{1+W(a)})(b)=a+\log(\mathcal{L}_{\exp(a)})(b)=ab.$$
    Recall that this is precisely the classical correspondence between two-sided braces and Jacobson radical rings \cite{rump2007braces}.
\end{example}

\begin{theorem}\label{theorem: correspondence p}
    Let $p^n$ be a prime power and $k<p$. There is a bijective correspondence between post-Lie rings of size $p^n$ and $L$-nilpotency class $k$, and skew braces of size $p^n$ and $L$-nilpotency class $k$.
\end{theorem}
\begin{proof}
    Let $(\fa,\tr)$ be a post-Lie ring of size $p^n$ and $L$-nilpotency class $k$. With its canonical filtration $(\fa,\tr)$ is a Lazard post-Lie ring since both $\fa$ and $\fa^\circ$ are $p$-groups with $\fa_p=\fa^\circ_p=\{0\}$. By \cref{theorem: final correspondence} we obtain a Lazard skew brace of size $p^n$ and $L$-nilpotency class $k$,

    Conversely, let $(A,\cdot,\circ)$ be a skew brace of size $p^n$ and $L$-nilpotency class $k$. Then with its canonical filtration it becomes a Lazard skew brace since $(A,\cdot)$ and $(A,\circ)$ are $p$-groups and $A_p=(A,\circ)_p=\{1\}$. By \cref{theorem: final correspondence} we obtain a Lazard post-Lie ring of size $p^n$ and $L$-nilpotency class $k$.
\end{proof}
We are now ready to prove \cref{theorem: correspondence finite p}.
\begin{proof}[Proof of \cref{theorem: correspondence finite p}]
    Recall that any skew brace of order $p^k$ is left nilpotent by \cite[Proposition 4.4]{CSV19}. Also, if a skew brace or post-Lie ring of order $p^k$ is $L$-nilpotent, its $L$-nilpotency class clearly can not exceed $k$. The statement now follows from \cref{theorem: correspondence p}, keeping in mind the characterisation of $L$-nilpotency in \cref{lem: left nilpotent nilpotent type Lie,lem: left nilpotent nilpotent type brace}.
\end{proof}

Let $(A,\cdot,\circ)$ be a skew brace and $a\in A$. For an integer $n$ we denote the $n$-th power of $a$ in $(A,\cdot)$ by $a^n$ and its $n$-th power in $(A,\circ)$ is denoted by $a^{\circ n}$. The following result extends Proposition 15 and Lemma 17 of \cite{SmoktunowiczOnThePassage2022preprint}.
\begin{proposition}
    Let $(A,\cdot,\circ)$ be a Lazard skew brace and let $n$ be an integer, then $\{a^n\mid a\in A\}=\{a^{\circ n}\mid a\in A\}$ and $\{a\in A\mid a^n=1\}=\{a\in A\mid a^{\circ n}=1\}$, and these sets are Lazard ideals of $(A,\cdot,\circ)$.
\end{proposition}
\begin{proof}
    We prove the statement for the first set, the second part is analogous. Let $(\fa,\tr)=\bL(A,\cdot,\circ)$. Then it is straightforward to see that $n\fa$ is a Lazard ideal of $(\fa,\tr)$, hence it is also a Lazard ideal of $(A,\cdot,\circ)$ by \cref{theorem: final correspondence}. Since $(A,\cdot)=\Laz(\fa)$, we find $n\fa=\{a^n\mid a\in A\}$. Also, as $W:\Laz(\fa^\circ)\to (A,\circ)$ is a group isomorphism we find $W(n\fa)=\{a^{\circ n}\mid a\in A\}$, but since $W$ maps any Lazard sub post-Lie ring to itself, we conclude that $\{a^n\mid a\in A\}=\{a^{\circ n}\mid a\in A\}$.
\end{proof}
\begin{proposition}
    Let $(\fa,\tr)$ be a Lazard post-Lie ring. Then the fix of $(\fa,\tr)$ coincides with that of $\bS(\fa,\tr)$, and similarly for the socle and annihilator. In particular, if $\fa$ is nilpotent, then $(\fa,\tr)$ is right nilpotent if and only if $\bS(\fa,\tr)$ is right nilpotent.
\end{proposition}
\begin{proof}
    The first part of the statement follows from the fact that $W(a)=a$ whenever $a\tr a=0$ or $\lambda_a(a)=a$, which in particular is the case when $a$ is contained in the fix or socle. The second part of the statement now follows from Lemmas 5.3 and 5.4 of \cite{acri2020retractability} and the analogous statement for post-Lie rings.
\end{proof}
\section{Formal differentiation through roots of unity}\label{section: differentiation}
Let $n\geq 1$. Recall that a \emph{primitive $n$-th root of unity} of a ring $R$ is an element $\xi\in R$ such that $\xi^n=1$ and $\xi^k-1$ not a zero divisor for $1\leq k<n$. It is well known that the finite field of prime power order $\F_{p^n}$ has a primitive $p^n-1$-th root of unity. Also, $\Z/p^n\Z$ always contains a primitive $p-1$-th root of unity.

Let $R$ be a commutative ring and let $M,N$ be $R$-modules. We say that a map $f:M\to N$ is a \textit{homogeneous polynomial of degree $k$} if $f(rm)=r^kf(m)$ for all $r\in R$, $m\in M$. A map $f:M\to N$ is polynomial of degree less than $k$ if it can be written as a sum $f=\sum_{i=0}^{k-1}f_i$ where $f_i$ is a homogeneous polynomial of degree $i$.
\begin{lemma}\label{lem: formal differential}
    Let $R$ be a ring and let $M,N$ be $R$-modules. Let $\xi\in R$ be a primitive $n$-th root of unity where $n$ is invertible in $R$ and let $f:M\to N$ be a polynomial of degree less than $n$, with homogeneous decomposition $f=\sum_{i=0}^{n-1}f_i$. Then $$\frac{1}{n}\sum_{j=0}^{n-1}\xi^{jk}f(\xi^{-j}m)=f_k(m)$$ for all $0\leq k<n$ and $m\in M$.
\end{lemma}
\begin{proof}
    Evaluating the formula yields
    \begin{align*}
        \frac{1}{n}\sum_{j=0}^{n-1}\xi^{jk}f(\xi^{-j}m)&=\frac{1}{n}\sum_{i=0}^{n-1}\sum_{j=0}^{n-1}\xi^{j(k-i)}f_i(m)=f_k(m)
    \end{align*}
    since $\frac{1}{n}\sum_{j=0}^{n-1}\xi^{j(k-i)}$ equals $0$ if $i\neq k$ and it equals $1$ if $i=k$.
\end{proof}
For $k=1$, it makes sense to see the above manipulation as taking the directional derivative of $f$ through $m$ and subsequently evaluating at $0$. Therefore, in the above setting, it seems justified to introduce the notation
$$df:M\to N:m\mapsto \frac{1}{n}\sum_{i=0}^{n-1}\xi^{i}f(\xi^{-i}m).$$
Although the analogy is not explicitly mentioned, this technique plays a prominent role in \cite{Smoktunowicz2022AAlgebras,Smoktunowicz2022OnRings}. The following proposition extends \cite[Theorem 12]{Smoktunowicz2022OnRings}. 

\begin{proposition}\label{prop: lambda through formal differential}
    Let $p^k$ be a prime power, $\xi\in \Z$ a primitive $(p-1)$-th root of unity modulo $p^k$. Let $(A,\cdot,\circ)$ be a filtered skew brace of order $p^k$ such that $A_p=\{1\}$ and such that $A_i*A_j\subseteq A_{i+j}$ for all $i,j\geq 1$. Let $(\fa,\tr)=\bL(A,\cdot,\circ)$, then $$a\tr b=\frac{1}{p-1}\sum_{i=0}^{p-2}\xi^i\lambda_{\xi^{-i}a}(b),$$ where the sum is taken in $\Laz^{-1}(A,\cdot)$.
\end{proposition}
\begin{proof}
     It is clear that $(A,\cdot,\circ)$ with the given filtration is a Lazard skew brace. By \cref{theorem: final correspondence} we know that $(A,\cdot,\circ)=\bS(\fa,\tr)$ for some Lazard post-Lie ring $(\fa,\tr)$. In particular, $\fa$ has a natural structure of a $\Z/p^k\Z$-module since $|\fa|=p^k$. Since $W(\fa_i)=\fa_i$, it follows from \cref{lem: lazard derivations and auto} that the assumption on the filtration translates to $\fa_i\tr \fa_j\subseteq \fa_{i+j}$ for all $i,j\geq 1$.
     
     Consider the map $\lambda:\fa\to \End(\fa):a\mapsto \lambda_a$. We claim that $\lambda$ is polynomial of degree less than $p-1$ and that moreover its degree 1 component is $\mathcal{L}_a$. Indeed, from \cref{prop: group of flows} we find 
     $$\lambda_{ta}=\exp(\mathcal{L}_{W(ta)})=\sum_{k=0}^\infty  \sum_{i=1}^\infty \frac{1}{k!i!}\mathcal{L}^k_{\mathcal{L}_{ta}^{i-1}(ta)}=\sum_{k=0}^\infty  \sum_{i=1}^\infty \frac{1}{k!i!}t^{ki}\mathcal{L}^k_{\mathcal{L}_{a}^{i-1}(a)}.$$
     Clearly $\mathcal{L}^{i-1}_{a}(a)\in \fa_{i}$. By the earlier mentioned condition on the filtration it follows that $\mathcal{L}^k_{\mathcal{L}_{a}^{i-1}(a)}(b)\in A^{ki+1}$ for all $b\in \fa$, so $\mathcal{L}^k_{\mathcal{L}_{a}^{i-1}(a)}=0$ for $ki\geq p-1$. Therefore, $\lambda$ is polynomial of degree less than $p-1$ and its degree 1 component is $\mathcal{L}_a$.
     From \cref{lem: formal differential} we obtain $d\lambda(a)=\mathcal{L}_a$ which concludes the proof.
\end{proof}

Note that a filtration as in \cref{prop: lambda through formal differential} exists precisely if $A^{\{p\}}=\{1\}$ where $A^{\{1\}}=A$ and $A^{\{k+1\}}$ is the subgroup of $(A,\cdot)$ generated by all $a*b$ and $aba^{-1}b^{-1}$ where $a\in A^{\{i\}}$, $b\in A^{\{k+1-i\}}$, $1\leq i \leq  k$. For braces this is precisely the chain of ideals $A^{[k]}$ used to define strong nilpotency as in \cite{smoktunowicz2018engel}. 

Remark that even when $R$ has characteristic $p^n$, the degree of the polynomial functions considered in \cref{lem: formal differential} are not directly bounded by $p$. On the other hand, in the formulae for $\exp$ and $\log$ no terms of degree $p$ or more are allowed. 
It is therefore natural to ask whether the nilpotency condition in \cref{prop: lambda through formal differential} can be relaxed under the assumption that an appropriate root of unity is present. More precisely, the following question is natural to ask:
\begin{question}
    Let $(A,+,\circ)$ be an $\F_{p^n}$-brace (as in \cite[Definition 2]{Catino2019SkewAnnihilator}) such that $A^{[p^n]}=1$ and let $\xi\in \F_{p^n}$ be a primitive $(p^n-1)$-th root of unity. Does the operation
    $$a\tr b=\frac{1}{p^n-1}\sum_{i=0}^{p^n-2}\xi^i\lambda_{\xi^{-i}a}(b),$$ 
    yield a pre-Lie $\F_{p^n}$-algebra structure on the $\F_{p^n}$-algebra $(A,+)$?
\end{question}

\section*{Acknowledgements}
The author was supported by Fonds Wetenschappelijk Onderzoek – Vlaanderen [1160524, G004124N], and by the Vrije Universiteit Brussel [OZR3762].

The author would also like to express his gratitude to Agata Smoktunowicz for some interesting discussions which contributed to the clarity of the exposition in this paper.
\bibliographystyle{abbrv}
\bibliography{Bibliography.bib}
\end{document}